\documentclass[12pt,reqno]{amsart}
\usepackage[utf8]{inputenc}

\usepackage{amsfonts, amssymb, amsmath, amsthm, mathtools, thmtools}
\usepackage{mathrsfs} % \mathscr
\usepackage{latexsym}% utilisation des symboles LaTeX pour avoir un beau LaTeX
\usepackage{enumerate}
\usepackage{multicol}
\usepackage{times}% font
\usepackage{verbatim}
\usepackage{tikz-cd}
\usepackage{fullpage}
\usepackage{here}
\usepackage{url}
\usepackage{csquotes}
\usepackage{placeins}
\usepackage{epic}
\usepackage{graphicx}
\usepackage{epstopdf}
\usepackage{pstricks}
\usepackage{pst-plot}
\usepackage{tikz}
\usepackage{xcolor}
\usepackage{graphicx}
\usepackage{subcaption}

\usetikzlibrary{shapes.geometric, positioning, calc}

\input xypic
\xyoption{all}
\xyoption{poly}

\usepackage[colorlinks=true]{hyperref}
\hypersetup{citecolor=blue, linkcolor=blue}

\usepackage[capitalise]{cleveref}

\crefname{equation}{}{}
\crefname{figure}{{\sc Figure}}{{\sc Figure}}
\crefname{subsection}{Subsection}{Subsections}

\usetikzlibrary{matrix,arrows,decorations.pathmorphing}

\usepackage[euler]{textgreek}
\usepackage{tikz,tikz-cd}
\usepackage[colorinlistoftodos, textwidth = 2.3cm]{todonotes}

\newtheorem{theorem}{Theorem}[section]
\newtheorem{proposition}[theorem]{Proposition}
\newtheorem{lemma}[theorem]{Lemma}
\newtheorem{corollary}[theorem]{Corollary}
\newtheorem{conjecture}[theorem]{Conjecture}

\newtheorem*{claim*}{Claim}

\theoremstyle{definition}

\newtheorem{remark}[theorem]{Remark}

\newcommand{\F}{{\mathbb F}}

\usepackage{tabstackengine}
\stackMath

\numberwithin{equation}{section} 
\numberwithin{figure}{section}
\numberwithin{table}{section}

\title{Multiplicative irreducibility of shifted multiplicative subgroups in the extremal case}
\author{Semin Yoo}
\address{Discrete Mathematics Group \\ Institute for Basic Science \\ 55 Expo-ro Yuseong-gu, Daejeon 34126 \\ South Korea}
\email{syoo19@ibs.re.kr}

\begin{document}

\begin{abstract}
In a recent breakthrough, Kalmynin proved a conjecture of S\'ark\"ozy on additive irreducibility of the set of quadratic residues in a prime field. More recently, Kim, Yip, and Yoo initiated the study of a multiplicative analogue of the conjecture for shifted multiplicative subgroups.
Specifically, they showed that for an odd prime $p$, a proper multiplicative subgroup $G$ of $\F_p^*$, and $\lambda\in G$, there do not exist sets $A,B\subseteq \F_p^*$ with $|A|,|B|\ge 2$ such that $AB=(G-\lambda)\setminus\{0\}$.
In this paper, when $\lambda \in \F_p^* \setminus G$, we completely resolve this problem in the equality case from a Stepanov bound in a prime field.
\end{abstract}

\maketitle

\section{Introduction}

Throughout the paper, let $p$ be a prime, $\F_p$ the finite field with $p$ elements, and $\F_p^*=\F_p\setminus \{0\}$. For two subsets $A,B$ of $\F_p$, we define their \emph{sumset} $A+B=\{a+b: a \in A, \ b\in B\}$ and \emph{product set} $AB=\{ab:a\in A, \ b \in B\}$. 
We denote by $\mathcal R_p$ the set of nonzero squares in $\F_p$ following \cite{LS17, S12, Kalmynin}.

The interaction between additive and multiplicative structures is a fundamental theme in arithmetic combinatorics. A basic question is whether a set with a strong arithmetic structure can be expressed as a nontrivial sumset or product set.
A celebrated conjecture of S\'ark\"ozy \cite{S12} asserts that, for every sufficiently large prime $p$, the set $\mathcal R_p$ cannot be represented in the form $\mathcal R_p=A+B$ with subsets $A,B\subseteq\F_p$ satisfying $|A|,|B|\geq 2$.
Hanson and Petridis \cite{HP} made progress on the conjecture by applying Stepanov's method \cite{S69}, proving that it holds for almost all primes.
In a recent breakthrough, Kalmynin~\cite{Kalmynin} completely resolved the conjecture by combining ingenious new ideas with the methods building on the work of Hanson and Petridis \cite{HP}.
We refer the reader to \cite{S12, S13, S14, S16, LS17, S20, KYY, Y24, Y25, Kalmynin} for related work on this conjecture.

In this paper, we study a multiplicative analogue of this conjecture. This problem was also proposed by S\'{a}rk\"{o}zy.

\begin{conjecture}[{S\'{a}rk\"{o}zy \cite{S14}}]\label{conj:M_S}
Let $p$ be a large enough prime. Then for each $\lambda\in \F_p^*$, the set $(\mathcal{R}_p-\lambda)\setminus\{0\}$ has no nontrivial multiplicative decomposition, that is,  there are no two subsets $A,B$ of $\F_p$ with $|A|,|B|\geq 2$, such that \[AB=(\mathcal{R}_p-\lambda)\setminus\{0\}.\]
\end{conjecture}

Later, Kim, Yip, and Yoo formulated a generalized version of \cref{conj:M_S}.
\begin{conjecture}[{\cite[Conjecture 1.9]{KYY}}]\label{conj:MS2}
Let $d \geq 2$ be fixed. Let $q \equiv 1 \pmod d$ be a sufficiently large prime power. Let $G$ be the multiplicative subgroup of $\F_q$ of index $d$. Then 
for each $\lambda\in\F_q^*$, the set $(G-\lambda)\setminus\{0\}$
has no nontrivial multiplicative decomposition.
\end{conjecture}

In the same paper, Kim, Yip, and Yoo \cite[Theorem 1.11]{KYY} obtained partial progress on \cref{conj:M_S} using Stepanov's method. A key ingredient in their work is the following estimate.

\begin{lemma}[{\cite[Theorem 1.1]{KYY}}]\label{thm: stepanovea}
Let $p$ be a prime and let $G$ be a proper multiplicative subgroup of $\F_p$. Let $A,B \subseteq \F_p^*$ and $\lambda \in \F_p^*$. 
If $AB+\lambda \subseteq G \cup \{0\}$, then 
\begin{equation}\label{eq:ineq1}
|A||B| \leq |G|+|B \cap (-\lambda A^{-1})|+|A|-1.
\end{equation}
Moreover, when $\lambda \in G$, we have a stronger upper bound:
\begin{equation}\label{eq:ineq2}
|A||B| \leq |G|+|B \cap (-\lambda A^{-1})|-1.
\end{equation}
\end{lemma}

More recently, Kim, Yip, and Yoo \cite[Theorem 1.10]{KYY26} proved \cref{conj:MS2} under the additional condition that the shift belongs to the subgroup.
\begin{theorem}[{\cite[Theorem 1.10]{KYY26}}]
Let $p$ be an odd prime, let $G$ be a proper multiplicative subgroup of $\F_p^*$, and let $\lambda\in G$. Then there do not exist sets $A,B\subseteq \F_p^*$ with $|A|,|B|\ge 2$ such that
\[
AB=(G-\lambda)\setminus\{0\}.
\]
\end{theorem}

In their proof, the assumption $\lambda \in G$ is crucial.
Indeed, if $AB=(G -\lambda) \setminus\{0\}$, then one can see that 
$B\cap(-\lambda A^{-1})=\varnothing$. 
By inequality \cref{eq:ineq2}, we have
\[
|G|-1=|AB|\leq |A||B|\leq |G|-1,
\]
and therefore \[
|A||B|=|G|-1.
\]
This exact cardinality relation makes it possible to apply the
techniques developed in \cite{KYY}.

On the other hand, the situation appears to be substantially more difficult when $\lambda\in \F_p^* \setminus G$.
Assuming $AB=(G-\lambda)\setminus\{0\}=G-\lambda$, one can see that
$B\cap(-\lambda A^{-1})=\varnothing$ and $A \cap (-\lambda B^{-1})=\varnothing$.
It follows from inequality \eqref{eq:ineq1} that
\[
|G|=|AB| \leq |A||B| \leq |G|+\min\{|A|,|B|\}-1.
\]
Thus $|A||B| = |G|+\min\{|A|,|B|\}-1$ is the extremal case of the Stepanov bound. (For more details, see \cref{sec:2}.)

It is also known that any hypothetical nontrivial multiplicative decomposition of a shifted subgroup must be essentially balanced. 
S\'ark\"ozy first obtained square root scale estimates in the case of shifted quadratic residues \cite[Theorem~1]{S14}. 
More generally, over prime fields, for an arbitrary proper multiplicative subgroup $G$, Shkredov proved $|A|,|B|=|G|^{1/2+o(1)}$ whenever $AB=(G-\lambda)\setminus\{0\}$ as $|G|\to\infty$; see \cite[Lemma~17]{S20}.
Kim, Yip, and Yoo also proved that, for every $\varepsilon>0$ and odd prime power $q$, if $G$ is the subgroup of $\F_q^*$ of index $d$, where $2\le d\le q^{1/2-\varepsilon}$, and $AB=(G-\lambda)\setminus\{0\}$ with $|A|,|B|\ge2$, then
\[
\frac{\sqrt q}{d}\ll_{\varepsilon}\min\{|A|,|B|\}\le\max\{|A|,|B|\}\ll_{\varepsilon}\sqrt q;
\]
see \cite[Proposition~3.4]{KYY}. In particular, when $d$ is fixed,
both $A$ and $B$ have order of magnitude $\sqrt q$.
Thus, any hypothetical decomposition is asymptotically balanced at the square root scale.

In this paper, we establish a further rigidity phenomenon by completely resolving the equality case of the Stepanov bound when $\lambda\in\F_p^*\setminus G$.
Our main theorem is the following.

\begin{theorem}\label{thm:delta-zero}
Let $p$ be an odd prime and let $G$ be a proper multiplicative subgroup of $\F_p$, and let  $\lambda \in \F_p^* \setminus G$. Suppose that $A,B\subseteq\F_p^*$ satisfy $|A|\geq 2$ and $|B|\geq 2$, and $AB=G-\lambda$ with $|A||B|=|G|+\min\{|A|,|B|\}-1$. Then we have 
\[p=11, \qquad G=\mathcal R_{11}, \qquad \text{and} \qquad \{|A|,|B|\}=\{2,3\}.\] 
In particular, if $p\neq 11$, then there do not exist subsets
$A,B\subseteq\F_p^*$ with $|A|,|B|\geq 2$ and $|A||B|=|G|+\min\{|A|,|B|\}-1$ such that
\[
AB=G - \lambda.
\]
\end{theorem}

The exceptional case in \cref{thm:delta-zero} does occur as mentioned in \cite{KYY26} when $G=\mathcal R_{11}=\{1,3,4,5,9\}$, $\lambda=2$, $A=\{1,7\}$, and $B=\{1,2,3\}$.

We briefly describe the proof of \cref{thm:delta-zero} as follows. 
Suppose for contradiction that $A,B\subseteq\F_p^*$ satisfy $|A|\geq 2$ and $|B|\geq 2$, and $AB=G-\lambda$ with $|A||B|=|G|+\min\{|A|,|B|\}-1$.
Write $n=|A|$ and $m=|B|$. We may assume that $n\le m$.
The argument is divided into three ranges, each requiring a different method. The first and third ranges lead to contradictions, whereas the second isolates the unique exceptional case $p=11$.

\begin{enumerate}
    \item If $m = n$, we use Newton's identities and properties of symmetric polynomials.
    \item If $n = 2 < m$, we use character sums and Weil's bound.
    \item If $m>n>2$, we use differential identities arising from analogues of Kalmynin's Relations X and Y.
\end{enumerate}

\subsection*{Notation}
Let $g(x)$ be a polynomial or a formal power series over $\F_p$. For an integer $i\geq 0$, we write $ [x^i]g(x) $ for the coefficient of $x^i$ in $g(x)$. We also write $ g(x)=O(x^n) $ if $x^n$ divides $g(x)$. 
\subsection*{Organization of the paper} 
In \cref{sec:2}, we review the multiplicative analogue of Hanson and Petridis polynomials, and recall symmetric polynomials.
In \cref{sec:3}, we establish analogues of Kalmynin's Relations X and Y, and derive useful lemmas from them. 
Finally, in \cref{sec:4}, we prove \cref{thm:delta-zero}.

\section{Preliminaries}\label{sec:2}

\subsection{A multiplicative analogue of Hanson--Petridis polynomials}

In this section, we review a multiplicative analogue of Hanson--Petridis polynomials developed by Kim, Yip, and Yoo \cite{KYY}, and derive useful facts.

Let $G$ be a proper multiplicative subgroup of $\F_p$ and $\lambda \in \F_p^* \setminus G$.
Let $A,B \subseteq \F_p^*$ with $|A|, |B| \geq 2$. 
Denote $A=\{a_1,a_2,\ldots, a_n\}$ with $|A|=n$ and $B=\{b_1,b_2,\ldots,b_m\}$ with $|B|=m$. 
Let $r:=|B\cap(-\lambda A^{-1})|$. We relabel the elements in $B$ so that
$b_1,\dots,b_r\in B\cap(-\lambda A^{-1})$ and $b_{r+1},\dots,b_m\in B\setminus(-\lambda A^{-1})$.

Recall from \cite[Section 4.1]{KYY} that there exists a unique solution $c_1,c_2,\ldots,c_n \in \F_p$ of the following system of equations:
\begin{equation} \label{system} 
\left\{
\TABbinary\tabbedCenterstack[l]{
\sum_{i=1}^n c_i=1\\\\
\sum_{i=1}^n c_i a_i^j=0,  \quad 1 \leq j \leq n-1
} .\right.    
\end{equation}

Moreover, one can obtain an explicit expression of the coefficients $c_1,c_2,\dots,c_n$ (see \cite[Lemma 2.1]{KYY26}):
\[
c_i=\frac{(-1)^{n-1}a_1a_2 \cdots a_n}{a_i \cdot \prod_{j\neq i}(a_i-a_j)}.   
\]
Using this expression, Kim, Yip, and Yoo obtained the following lemma. 

\begin{lemma}[{\cite[Lemma 2.1]{KYY26}}]\label{lem:GF}
The coefficients $c_1,\dots,c_n$ satisfy
\[
\sum_{i=1}^n \frac{c_i a_i^n}{1-a_ix}
= (-1)^{n-1}\Big(\prod_{i=1}^n a_i\Big)\cdot\frac{1}{\prod_{i=1}^n(1-a_ix)}.
\]
\end{lemma}

Next, following \cite[Section 4.1]{KYY}, consider the following auxiliary polynomial
\[
f(x)=-\lambda^{n-1}+\sum_{i=1}^n c_i (a_ix+\lambda)^{n-1+|G|}\in \F_p[x].    
\]
Kim, Yip, and Yoo showed that $f$ is a non-zero polynomial, each of $b_1,\dots,b_r$ is a root of $f$ with multiplicity at least $n-1$,
and each of $b_{r+1},\dots,b_m$ is a root of $f$ with multiplicity at least $n$.
It follows that
\[
r(n-1)+(m-r)n=mn-r \le \deg f \le n-1+|G|.
\]

Now, suppose that there exist subsets $A,B\subseteq\F_p^*$ with $|A|=n\geq 2 $ and $|B|=m\geq 2$ such that
\[
AB=(G-\lambda)\setminus\{0\}=G-\lambda.
\]
The second equality holds since $\lambda\in \F_p^* \setminus G$, and so $0\notin G-\lambda$.

We first observe that $B\cap(-\lambda A^{-1})=\varnothing$, so $r=0$.
Indeed, if $b\in B\cap(-\lambda A^{-1})$, then $b=-\lambda a^{-1}$ for some $a\in A$, and thus $ab=-\lambda\in AB$.
However, $-\lambda\notin G-\lambda$, since an equality
$g-\lambda=-\lambda$ with $g\in G$ would imply $g=0$, contradicting
$G\subseteq\F_p^*$.

By symmetry, we also have $A \cap (-\lambda B^{-1})=\varnothing$, and we may assume that $n \le m$.
Since $|AB|=|G|$, we have $|G|\leq nm$.
By inequality \cref{eq:ineq1} and the fact that $r=0$, we have
\[
|G| \leq nm \leq |G|+\min\{n,m\}-1=|G|+n-1.
\]

Put
\[
P(x):=\prod_{b\in B}(x-b).
\]
Since every $b\in B$ is a root of $f$ with multiplicity at least $n$, we have $P(x)^n\mid f(x)$, and thus there exists a polynomial $R(x)\in\F_p[x]$ such that
\begin{equation}\label{eq:factor}
f(x)=P(x)^nR(x).
\end{equation}

We now specialize to the case $|G|+n-1=nm$.
Then $\deg R=0$, so equation \eqref{eq:factor} becomes
\begin{equation}\label{eq:delta0-factor}
f(x)=C P(x)^n
\end{equation}
for some $C\in\F_p^*$.
Note that, since $\lambda\in \F_p^* \setminus G$, we have $\lambda^{|G|}\neq 1$.
Therefore,
\[
f(0) =-\lambda^{n-1}+\lambda^{|G|+n-1}\sum_{i=1}^n c_i
=-\lambda^{n-1}+\lambda^{|G|+n-1}
=\lambda^{n-1}\bigl(\lambda^{|G|}-1\bigr) \neq 0.
\]
By system of equations \cref{system}, the coefficients of $x^j$ in $f(x)$ vanish for $1\leq j\leq n-1$.
Equivalently, $f(x)=f(0)+O(x^n)$.
This, together with $f(0)\neq 0$, gives, as an identity of formal power series at $x=0$,
\[
\frac{f'(x)}{f(x)}=O(x^{n-1}).
\]
On the other hand, differentiating equation \eqref{eq:delta0-factor} gives
\[
\frac{f'(x)}{f(x)} = n\frac{P'(x)}{P(x)} = n\sum_{b\in B}\frac{1}{x-b}.
\]
Since $b\neq 0$ for every $b\in B$, we may expand
\[
\frac{1}{x-b} = -\sum_{r=0}^{\infty}b^{-(r+1)}x^r.
\]
Therefore,
\[
\frac{f'(x)}{f(x)} = -\sum_{r=0}^{\infty} \left(n\sum_{b\in B}b^{-(r+1)}\right)x^r.
\]
Since the coefficients of $x^0,\ldots,x^{n-2}$ in $f'/f$ vanish, and $n<p$, comparison of coefficients gives
\begin{equation}\label{eq:bgamma}
\sum_{b\in B}b^{-t}=0
\qquad\text{for }1\leq t\leq n-1.
\end{equation}

Furthermore,
\[
\frac{P'(x)}{P(x)} = \frac{1}{n}\frac{f'(x)}{f(x)} = O(x^{n-1}).
\]
Since $P(0)\neq 0$, it follows that 
\begin{equation}\label{eq:P-derivative-gaps}
P'(x)=O(x^{n-1}).  
\end{equation}
As $n<p$, this implies
\begin{equation}\label{eq:delta0-P-gap-middle}
P(x)=P(0)+O(x^n).
\end{equation}

\subsection{Symmetric polynomials}\label{sec: 2.1}
In this section, we recall standard families of symmetric polynomials and fix our notation. 
Throughout, $n\ge 1$ is an integer and $x_1,\dots,x_n$ are indeterminates. 

\begin{itemize}

    \item The \textit{elementary symmetric polynomials} are defined by
    \[ e_k(x_1,x_2,\dots,x_n)
        = \sum_{1\le i_1< i_2< \cdots <i_k\le n}
          x_{i_1}x_{i_2}\cdots x_{i_k} \]
    for $1\le k \le n$. We set $e_0(x_1,x_2,\dots,x_n)=1$ and $e_k(x_1,x_2,\dots,x_n)=0$ for $k>n$.

    \item The \textit{complete homogeneous symmetric polynomials} are defined by 
    $$ h_{k}(x_{1},x_{2},\dots ,x_{n})=\sum _{1\leq i_{1}\leq i_{2}\leq \cdots \leq i_{k}\leq n}x_{i_{1}}x_{i_{2}}\cdots x_{i_{k}}=\sum_{\substack{i_1+i_2+\cdots+i_n=k\\i_j\ge0}}x_1^{i_1}x_2^{i_2}\cdots x_n^{i_n}$$
    for $k\geq 1$, and set $h_0(x_1,\ldots,x_n):=1$.

    \item The \textit{power sum symmetric polynomials} are defined by
    \[ p_k(x_1,x_2,\dots,x_n)
        = \sum_{i=1}^n x_i^k\]
    for $k\ge 0$.
\end{itemize}

The following well-known identity relates elementary symmetric polynomials to powe sum symmetric polynomials.
\begin{lemma}[Newton's identities \cite{Stanley}]\label{lem:newton}
For each integer $k\ge 1$, we have
\begin{equation}\label{eq: ehp2}
ke_{k}(x_{1},x_{2},\dots ,x_{n})=\sum_{i=1}^{k}(-1)^{i-1}e_{k-i}(x_{1},x_{2},\dots ,x_{n})p_{i}(x_{1},x_{2},\dots ,x_{n}),
\end{equation}
where the identity holds whenever $n\geq k \geq 1$. 
\end{lemma}

\section{Analogues of Kalmynin's Relations $X$ and $Y$ for multiplicative subgroups}\label{sec:3}

In this section, we establish analogues of Kalmynin's Relations X and Y, and derive some useful lemmas. 
Throughout this section, let $G$ be a proper multiplicative subgroup of $\F_p$, and $\lambda\in \F_p^* \setminus G$.

Assume that there exist $A,B \subseteq \F_p^*$ such that $AB=G-\lambda$ with $|A|=n$ and $|B|=m$.
Denote $A=\{a_1,a_2,\ldots, a_n\}$ and $B=\{b_1,b_2,\ldots,b_m\}$ again.
We also assume $n \le m$ and $nm=|G|+n-1$, so $|G|=nm-n+1=n(m-1)+1$.
Fixing any $b\in B$, since multiplication by $b$ is injective and
$Ab\subseteq AB=G-\lambda$, we have $n=|A|\le |G|$.
Thus,
\[
N:=|G|+n-1\le 2|G|-1\le p-2,
\]
since $G$ is a proper subgroup of $\F_p^*$.
We use the notation
\[
(N)_k:=N(N-1)\cdots(N-k+1).
\]
Since $N\le p-2$, all the falling factorials used in this section are nonzero
in $\F_p$. Moreover, note that $m-1$, $n+1$, $n+2$, and $n(m-1)-1=|G|-2$ are all nonzero in $\F_p$.

Define
\[
J:=\left\{\frac{\lambda}{a}:a\in A\right\} \qquad  \text{and} \qquad Q(x):=\prod_{c\in J}(x+c).
\]
For $b\in B$, set
\[
L_1(b):=\sum_{\substack{b'\in B\\ b'\ne b}}\frac1{b-b'}, \quad L_2(b):=\sum_{\substack{b'\in B\\ b'\ne b}}\frac1{(b-b')^2}, \quad S_1(b):=\sum_{c\in J}\frac1{b+c}, \quad S_2(b):=\sum_{c\in J}\frac1{(b+c)^2}.
\]

\begin{lemma}[Relation $X$]\label{lem:relation-X}
For every $b\in B$,
\[
S_1(b)=\frac{n+1}{m-1}L_1(b).
\]
Equivalently,
\[
\sum_{c\in J}\frac1{b+c} = \frac{n+1}{m-1} \sum_{\substack{b'\in B\\ b'\ne b}}\frac1{b-b'}.
\]
\end{lemma}

\begin{proof}
Fix $b\in B$. From equation \eqref{eq:delta0-factor}, write $f(x)=(x-b)^n g_b(x)$, where
\[
\quad g_b(x)=C\prod_{\substack{b'\in B\\ b'\ne b}}(x-b')^n.
\]
Then
\begin{equation}\label{eq:relation-X-h-log}
\frac{g_b'(b)}{g_b(b)}=nL_1(b).
\end{equation}
On the other hand, $f^{(n)}(b)=n!g_b(b)$ and $f^{(n+1)}(b)=(n+1)!g_b'(b)$, so
\begin{equation}\label{eq:relation-X-der-ratio}
\frac{g_b'(b)}{g_b(b)}=\frac{f^{(n+1)}(b)}{(n+1)f^{(n)}(b)}.
\end{equation}
Let $N:=|G|+n-1$. By the definition of $f$, for $1\le k\le N$,
\[
f^{(k)}(x)=(N)_k\sum_{i=1}^n c_i a_i^k(a_i x+\lambda)^{N-k}.
\]
Since $a_i b+\lambda\in G$, we have $(a_i b+\lambda)^{|G|}=1$. Moreover, $N-n=|G|-1$. Thus
\[
f^{(n)}(b)=(N)_n\sum_{i=1}^n c_i\frac{a_i^n}{a_i b+\lambda} \qquad \text{and} \qquad f^{(n+1)}(b)=(N)_{n+1}\sum_{i=1}^n c_i\frac{a_i^{n+1}}{(a_i b+\lambda)^2}.
\]
Here, all denominators are nonzero since $a_ib+\lambda \in G \subset \F_p^*$ for every $1 \le i \le n$.
Substituting into \eqref{eq:relation-X-der-ratio} gives
\begin{equation}\label{eq:relation-X-ratio-before-GF}
\frac{g_b'(b)}{g_b(b)} = \frac{|G|-1}{n+1} \cdot \frac{\displaystyle\sum_{i=1}^n c_i\frac{a_i^{n+1}}{(a_i b+\lambda)^2}}
{\displaystyle\sum_{i=1}^n c_i\frac{a_i^n}{a_i b+\lambda}}.
\end{equation}
Put
\[
F(x):=\sum_{i=1}^n\frac{c_i a_i^n}{1-a_i x}.
\]
By \cref{lem:GF}, we have
\[
F(x)=(-1)^{n-1}\left(\prod_{i=1}^n a_i\right)\frac{1}{\prod_{i=1}^n(1-a_i x)}.
\]
Differentiating the first expression for $F$, we obtain
\[
F'(x)=\sum_{i=1}^n\frac{c_i a_i^{n+1}}{(1-a_i x)^2}.
\]
On the other hand, taking the derivative of the second
expression gives
\begin{equation}\label{eq:second}
\frac{F'(x)}{F(x)}=\sum_{i=1}^n\frac{a_i}{1-a_i x}.
\end{equation}
We now evaluate at $x=-b/\lambda$. Then
we have
\[
F\left(-\frac{b}{\lambda}\right)=\lambda \sum_{i=1}^n c_i\frac{a_i^n}{a_i b+\lambda} \qquad \text{and} \qquad F'\left(-\frac{b}{\lambda}\right) = \lambda^2 \sum_{i=1}^n c_i\frac{a_i^{n+1}}{(a_i b+\lambda)^2}.
\]
Therefore,
\[
\frac{\displaystyle\sum_{i=1}^n c_i\frac{a_i^{n+1}}{(a_i b+\lambda)^2}}
{\displaystyle\sum_{i=1}^n c_i\frac{a_i^n}{a_i b+\lambda}}
=\frac{1}{\lambda}\frac{F'(-b/\lambda)}{F(-b/\lambda)} 
= \sum_{i=1}^n \frac{a_i}{a_ib+\lambda}= \sum_{c\in J}\frac{1}{b+c} =  S_1(b).
\]
The second equality is from equation \cref{eq:second}.
Combining this with equations \eqref{eq:relation-X-h-log} and \eqref{eq:relation-X-ratio-before-GF}, we obtain
\[
nL_1(b)=\frac{|G|-1}{n+1}S_1(b).
\]
Since $|G|-1=n(m-1)$, this completes the proof.
\end{proof}

\begin{remark}\label{rem:X-logarithmic-form}
Recall that
\[
P(x)=\prod_{b\in B}(x-b) \qquad\text{and}\qquad Q(x)=\prod_{c\in J}(x+c).
\]
Relation $X$ can be written in the form as
\begin{equation}\label{eq:X-logarithmic-form}
\frac{Q'(b)}{Q(b)} = \frac{n+1}{m-1}\frac{P''(b)}{2P'(b)}
\qquad \text{for every }b\in B
\end{equation}
since
\[
\frac{Q'(b)}{Q(b)}=S_1(b) \qquad \text{and} \qquad \frac{P''(b)}{2P'(b)}=L_1(b).
\]
\end{remark}

\begin{lemma}[Relation $Y$]\label{lem:relation-Y}
For every $b\in B$,
\[
S_1(b)^2+S_2(b) = \frac{(n+1)(n+2)}{(m-1)(n(m-1)-1)}\left(nL_1(b)^2-L_2(b)\right).
\]
Equivalently,
\begin{align}
&\left(\sum_{c\in J}\frac1{b+c}\right)^2 + \sum_{c\in J}\frac1{(b+c)^2} \nonumber \\
&\qquad= \frac{(n+1)(n+2)}{(m-1)(n(m-1)-1)} \left( n\left(\sum_{\substack{b'\in B\\ b'\ne b}}\frac1{b-b'}\right)^2 - \sum_{\substack{b'\in B\\ b'\ne b}}\frac1{(b-b')^2}
\right). \nonumber
\end{align}
\end{lemma}

\begin{proof}
Fix $b\in B$, and write again $f(x)=(x-b)^n g_b(x)$, where
\[
g_b(x)=C\prod_{\substack{b'\in B\\ b'\ne b}}(x-b')^n.
\]
Then
\begin{equation}\label{eq:relation-Y-h-second-log}
\frac{g_b'(b)}{g_b(b)}=nL_1(b) \qquad \text{and} \qquad  
\frac{g_b''(b)}{g_b(b)}=n\left(nL_1(b)^2-L_2(b)\right).
\end{equation}

Now, we find another expression of $g_b''(b)/g_b(b)$.
Since $f^{(n)}(b)=n!g_b(b)$ and $f^{(n+2)}(b)=\binom{n+2}{2}n!g_b''(b)$, we also have
\[
\frac{g_b''(b)}{g_b(b)}=\frac{2f^{(n+2)}(b)}{(n+1)(n+2)f^{(n)}(b)}.
\]
As in the proof of \cref{lem:relation-X}, with $N=|G|+n-1$,
\[
f^{(n)}(b)=(N)_n\sum_{i=1}^n c_i\frac{a_i^n}{a_i b+\lambda}
\qquad \text{and} \qquad f^{(n+2)}(b)=(N)_{n+2}\sum_{i=1}^n c_i\frac{a_i^{n+2}}{(a_i b+\lambda)^3}.
\]
Thus
\begin{equation}\label{eq:relation-Y-ratio-before-h2}
\frac{g_b''(b)}{g_b(b)} = \frac{2(|G|-1)(|G|-2)}{(n+1)(n+2)}
\cdot \frac{\displaystyle\sum_{i=1}^n c_i\frac{a_i^{n+2}}{(a_i b+\lambda)^3}} {\displaystyle\sum_{i=1}^n c_i\frac{a_i^n}{a_i b+\lambda}}.
\end{equation}

Next, we simplify the term on the right-hand side of equation \cref{eq:relation-Y-ratio-before-h2}.
Put
\[
F(x):=\sum_{i=1}^n\frac{c_i a_i^n}{1-a_i x}.
\]
Differentiating equation \cref{eq:second}, we obtain
\[
\frac{F''(x)}{F(x)}-\left(\frac{F'(x)}{F(x)}\right)^2=\sum_{i=1}^n\frac{a_i^2}{(1-a_i x)^2}.
\]
Thus, this, together with equation \cref{eq:second}, gives
\[
\frac{F''(x)}{2F(x)}=\frac12\left[\left(\sum_{i=1}^n\frac{a_i}{1-a_i x}\right)^2+\sum_{i=1}^n\frac{a_i^2}{(1-a_i x)^2}\right]=h_2\left(\frac{a_1}{1-a_1x},\ldots,\frac{a_n}{1-a_nx}\right).
\]
Now put $z_i:=a_i/(a_i b+\lambda)$.
Evaluating at $x=-b/\lambda$, we have
\[
\frac{a_i}{1+a_i b/\lambda}
=
\lambda z_i.
\]
Since $h_2$ is homogeneous of degree $2$, it follows that
\begin{equation}\label{eq:F-second-ratio}
\frac{F''(-b/\lambda)}{2F(-b/\lambda)}
=
\lambda^2 h_2(z_1,\ldots,z_n).
\end{equation}
On the other hand,
\[
F\left(-\frac b\lambda\right)=\lambda\sum_{i=1}^n c_i\frac{a_i^n}{a_i b+\lambda} \qquad \text{and} \qquad F''\left(-\frac b\lambda\right)
= 2\lambda^3\sum_{i=1}^n c_i\frac{a_i^{n+2}}{(a_i b+\lambda)^3}.
\]
Substituting these identities into equation \eqref{eq:F-second-ratio}, we obtain
\begin{equation}\label{eq:relation-Y-h2-ratio}
\frac{\displaystyle\sum_{i=1}^n c_i\frac{a_i^{n+2}}{(a_i b+\lambda)^3}}
{\displaystyle\sum_{i=1}^n c_i\frac{a_i^n}{a_i b+\lambda}} = h_2(z_1,\ldots,z_n).
\end{equation}
Since
\[
h_2(z_1,\ldots,z_n) =\frac12\left(\left(\sum_{i=1}^n z_i\right)^2+\sum_{i=1}^n z_i^2\right),
\]
we get
\[
h_2(z_1,\ldots,z_n)=\frac12\left(S_1(b)^2+S_2(b)\right).
\]
Combining this with equations \eqref{eq:relation-Y-h-second-log}, \eqref{eq:relation-Y-ratio-before-h2}, and \eqref{eq:relation-Y-h2-ratio}, we obtain
\[
n\left(nL_1(b)^2-L_2(b)\right) = \frac{(|G|-1)(|G|-2)}{(n+1)(n+2)} \left(S_1(b)^2+S_2(b)\right).
\]
Finally, using $|G|-1=n(m-1)$ and $|G|-2=n(m-1)-1$, the lemma follows.
\end{proof}

Using Relation X, we have the following lemma.

\begin{lemma}\label{lem:X-global}
We have
\[
\frac{n+1}{m-1} Q(x)P''(x)-2Q'(x)P'(x) =-m(n-1)x^{n-2}P(x).
\]
\end{lemma}

\begin{proof}
By equation \cref{eq:X-logarithmic-form}, we have
\[
P(x)\mid \frac{n+1}{m-1} Q(x)P''(x)-2Q'(x)P'(x).
\]
Let $\mathcal X(x):=\frac{n+1}{m-1} Q(x)P''(x)-2Q'(x)P'(x)$.
Since $\deg Q=n$ and $\deg P=m$, we have $\deg \mathcal X\le n+m-2$,
so $\deg(\mathcal X/P)\le n-2$.
On the other hand, equation \cref{eq:P-derivative-gaps} gives $\mathcal X(x)=O(x^{n-2})$.
Since $P(0)\ne0$, it follows that $\mathcal X(x)/P(x)=O(x^{n-2})$.
Thus $\mathcal X/P$ is a scalar multiple of $x^{n-2}$. 

Since $Q$ and $P$ are monic of degrees $n$ and $m$, respectively,
we have 
\[Q(x)P''(x)=m(m-1)x^{n+m-2}+\cdots \qquad  \text{and} \qquad  Q'(x)P'(x)=mnx^{n+m-2}+\cdots.\]
Thus $[x^{n+m-2}]\mathcal X =-m(n-1)$.
Comparing leading coefficients gives
\[
\frac{\mathcal X(x)}{P(x)} =-m(n-1)x^{n-2}.
\]
This proves the lemma.
\end{proof}

\begin{lemma}\label{lem:Y-divisibility}
Define
\begin{align*}
\mathcal R(x) :=&\ (n-1)(3mn+2m-n)Q'(x)^2 -(n+1)(3mn-n+1)Q(x)Q''(x) \\ 
&+m(n-1)(n+1)(n+2)x^{n-2}Q(x).  
\end{align*}
Then $P(x)\mid \mathcal R(x)$.
\end{lemma}

\begin{proof}
It is enough to show that $\mathcal R(b)=0$ for every $b\in B$.
Fix $b\in B$. Since $P$ has simple roots, $P'(b)\ne0$. Also $Q(b)\ne0$, since if $Q(b)=0$, then $b=-\lambda/a$ for some $a\in A$, and so $ab+\lambda=0$, a contradiction.

We find constants $\rho,\sigma,\tau$ such that
\[
\rho Q'(b)^2+\sigma Q(b)Q''(b)+\tau b^{n-2}Q(b)=0.
\]
After dividing by $Q(b)^2$, this is equivalent to finding a linear relation among $\eta^2$, $\xi$, and $b^{n-2}/Q(b)$,
where $\eta:=Q'(b)/Q(b)$ and $\xi:=Q''(b)/Q(b)$.

Put
\begin{equation}\label{eq:alphabeta}
\alpha:=\frac{n+1}{m-1} \qquad \text{and} \qquad \beta:=\frac{(n+1)(n+2)}{(m-1)(n(m-1)-1)}.
\end{equation}
\cref{lem:relation-X} and equation \cref{eq:X-logarithmic-form} give $\eta=S_1(b)=\alpha L_1(b)$, so $L_1(b)=\eta/\alpha$.
Moreover, by a derivative calculation from the definition of $Q$, we have $\left(Q'/Q\right)'(b)=-S_2(b)$. 
Since $\left(Q'/Q\right)'(b)=\xi-\eta^2$ due to chain rule, we have
\begin{equation}\label{eq:fromX2}
S_1(b)^2+S_2(b)=2\eta^2-\xi.
\end{equation}
Next, write $\omega:=P'''(b)/P'(b)$.
Since a derivative calculation from the definition of $P$ gives
\[
L_1(b)=\frac{P''(b)}{2P'(b)} \qquad  \text{and} \qquad L_2(b)=L_1(b)^2-\frac{P'''(b)}{3P'(b)},
\]
we have
\begin{equation}\label{eq:mid5}
L_2(b)=L_1(b)^2-\frac{\omega}{3}.
\end{equation}
Using equations \cref{eq:fromX2}, \cref{eq:mid5}, and \cref{lem:relation-Y} gives
\[
2\eta^2-\xi = S_1(b)^2+S_2(b) = \beta\left(nL_1(b)^2-L_2(b)\right) = \beta\left((n-1)L_1(b)^2+\frac{\omega}{3}\right).
\]
Thus, by $L_1(b)=\eta/\alpha$,
\begin{equation}\label{eq:mid4}
2\eta^2-\xi = \beta\left(\frac{n-1}{\alpha^2}\eta^2+\frac{\omega}{3}\right)
\quad \text{and so} \quad  \omega= \frac{3}{\beta}(2\eta^2-\xi) -\frac{3(n-1)}{\alpha^2}\eta^2.
\end{equation}

Note that we have $\alpha QP''-2Q'P'=-m(n-1)x^{n-2}P$ by \cref{lem:X-global}. Differentiating the equation, evaluating at $x=b$, and using $P(b)=0$, we get
\[
\alpha Q(b)P'''(b)+(\alpha-2)Q'(b)P''(b)-2Q''(b)P'(b) =-m(n-1)b^{n-2}P'(b).
\]
Dividing by $Q(b)P'(b)$, and using
\[
\frac{P''(b)}{P'(b)}=2L_1(b)=\frac{2\eta}{\alpha},
\]
we obtain
\begin{equation}\label{eq:mid2}
\alpha \omega +\frac{2(\alpha-2)}{\alpha}\eta^2 -2\xi + m(n-1)\frac{b^{n-2}}{Q(b)} =0.
\end{equation}
Substituting equation \cref{eq:mid4} into equation \cref{eq:mid2}, we have 
\begin{equation}\label{eq:mid}
\left( \frac{6\alpha}{\beta} -\frac{3(n-1)}{\alpha} +\frac{2(\alpha-2)}{\alpha} \right)\eta^2 + \left( -\frac{3\alpha}{\beta}-2 \right)\xi + m(n-1)\frac{b^{n-2}}{Q(b)} =0.
\end{equation}
Using equation \cref{eq:alphabeta}, a direct simplification gives
\[
\frac{6\alpha}{\beta} -\frac{3(n-1)}{\alpha} +\frac{2(\alpha-2)}{\alpha} = \frac{(n-1)(3mn+2m-n)}{(n+1)(n+2)} \qquad \text{and} \qquad 
-\frac{3\alpha}{\beta}-2 = -\frac{3mn-n+1}{n+2}.
\]
Therefore, multiplying equation \cref{eq:mid} by $(n+1)(n+2)Q(b)^2$, we obtain
\begin{align*}
(n-1)&(3mn+2m-n)Q'(b)^2 -(n+1)(3mn-n+1)Q(b)Q''(b)\\
&\qquad +m(n-1)(n+1)(n+2)b^{n-2}Q(b)=0.
\end{align*}
In other words, $\mathcal R(b)=0$.
\end{proof}

\section{Proof of \cref{thm:delta-zero}}\label{sec:4}

In this section, we prove \cref{thm:delta-zero}. 

\subsection{The case (1): $n=m$}

\begin{proposition}\label{prop:case1}
Let $p$ be an odd prime, let $G$ be a proper multiplicative subgroup of $\F_p$, and let $\lambda\in \F_p^* \setminus G$.
Let $A,B\subseteq\F_p^*$ with $|A|=|B|=n\geq 2$.
If $|G|+n-1=n^2$, then
\[
AB\neq G- \lambda.
\]
\end{proposition}

\begin{proof}
Suppose, for contradiction, that $AB=G-\lambda$.
Write $B=\{b_1,\ldots,b_n\}$.
Since $|AB|=|G|$, fixing any $b\in B$ and considering the injective map $A\rightarrow AB$, $a\mapsto ab$, we obtain $n\le |G|$. Thus
\[
n^2=|G|+n-1\le 2|G|-1\le p-2,
\]
where the last inequality follows since $G$ is a proper subgroup of $\F_p^*$. In particular, $n<p$.

By equation \eqref{eq:bgamma},
\[
\sum_{j=1}^n b_j^{-t}=0 \qquad \text{for }1\le t\le n-1.
\]
Then, we have $p_t(b_1^{-1},\ldots,b_n^{-1})=0$ for $1\le t\le n-1$.
\cref{lem:newton} gives
\[
e_1(b_1^{-1},\ldots,b_n^{-1})=\cdots = e_{n-1}(b_1^{-1},\ldots,b_n^{-1})=0.
\]
Thus,
\[
\prod_{j=1}^n(x-b_j^{-1})=x^n+(-1)^n e_n(b_1^{-1},\ldots,b_n^{-1}).
\]
In particular, for every $j=1,\ldots,n$, we have $(b_j^{-1})^n = (-1)^{n+1}e_n(b_1^{-1},\ldots,b_n^{-1})$, so all the elements $b_j^{-1}$ have the same $n$-th power.
Fix $b_1^{-1}$, and put
\[
H:=\{h\in\F_p^*:h^n=1\}.
\]
For every $j$, we have $\left(b_j^{-1}/b_1^{-1}\right)^n=1$, and so $b_j^{-1}/b_1^{-1}\in H$. Since the elements $b_1^{-1},\ldots,b_n^{-1}$ are distinct, this gives at least $n$ distinct elements of $H$. On the other hand, the polynomial $x^n-1$ has at most $n$ roots, so $|H|=n$ and $B^{-1}=b_1^{-1}H$.
Since $H^{-1}=H$, it follows that $B=\beta H$ for some $\beta\in\F_p^*$.
Thus, $AB=\beta AH$.
The set $AH$ is a union of cosets of $H$, and distinct cosets of $H$ are disjoint. Therefore,
\[
n=|H|\mid |AH|=|AB|.
\]
However,
\[
|AB|=|G|=n^2-n+1\equiv1\pmod n,
\]
which gives a contradiction. This proves the proposition.
\end{proof}

\subsection{The case (2): $n=2<m$}

\begin{proposition}\label{prop:case2}
Let $p$ be an odd prime, let $G=\mathcal R_p$, and let $\lambda\in \F_p^* \setminus G$.
Suppose that $A,B\subseteq\F_p^*$ with $|A|=2$ and $|B|=m >2$ satisfy $AB=G- \lambda$.
Assume $|G|+1=2m$. Then $p=11$.
\end{proposition}

\begin{proof}
Since $|\mathcal R_p|=(p-1)/2$, we have $2m=|\mathcal R_p|+1=(p+1)/2$.
Thus $p=4m-1$.
In particular, $p\equiv3\pmod4$.

We may assume $A=\{1,r\}$ with $r\ne1$.
Since $\lambda\in \F_p^* \setminus \mathcal R_p$, multiplication by $\lambda^{-1}$ sends $\mathcal R_p$ to the set $\mathcal N_p$ of quadratic nonresidues. Thus $AB+\lambda=\mathcal R_p$ is equivalent to 
\begin{equation}\label{eq:nonresidues} (1+B/\lambda)\cup(1+r B/\lambda)=\mathcal N_p.
\end{equation}
Note that both $1+B/\lambda$ and $1+rB/\lambda$ have size $m$, while $|\mathcal N_p|=\frac{p-1}{2}=2m-1$.
Thus
\begin{equation}\label{eq:n2-T-intersection-one}
|(1+B/\lambda)\cap(1+rB/\lambda)|=1.
\end{equation}

Let $\chi$ denote the quadratic character of $\F_p$ with $\chi(0)=0$. Define
\[
T_r:=\{t\in\F_p:\chi(1+t)=\chi(1+rt)=-1\}.
\]
From equation \eqref{eq:nonresidues}, we have $B/\lambda\subseteq T_r$.

Now, we compute $|T_r|$. Note that
\[
\frac14\sum_{t\in\F_p}(1-\chi(1+t))(1-\chi(1+rt))=\frac{p-\chi(r)}4
\] 
using
\[
\sum_{t\in\F_p}\chi(1+t)=\sum_{t\in\F_p}\chi(1+rt)=0 \qquad \text{and} \qquad \sum_{t\in\F_p}\chi((1+t)(1+rt))=-\chi(r).
\]
Thus we have
\begin{equation}\label{eq:n2-Tr-count}
|T_r| = \frac{p-\chi(r)}{4} - \frac{1-\chi(1-r)}{4} - \frac{1-\chi(1-r^{-1})}{4}.
\end{equation}
Indeed, since $r\ne1$, we have that $-1$ and $-r^{-1}$ are distinct.
At the two points $t=-1$ and $t=-r^{-1}$, this main term counts possible zero values. Subtracting those two spurious contributions gives equation \eqref{eq:n2-Tr-count}.

Since $|B/\lambda|=m=(p+1)/4$ and since $B/\lambda\subseteq T_r$, equation \eqref{eq:n2-Tr-count} implies
\[
\frac{p+1}{4} \le \frac{p-\chi(r)-2+\chi(1-r)+\chi(1-r^{-1})}{4}.
\]
Equivalently, $3\le -\chi(r)+\chi(1-r)+\chi(1-r^{-1})$.
Since each term on the right-hand side is at most $1$, equality must hold throughout. Consequently,
\begin{equation}\label{eq:n2-r-character-conditions}
\chi(r)=-1, \qquad \chi(1-r)=1, \qquad \text{and} \qquad \chi(1-r^{-1})=1.
\end{equation}
Thus $|T_r| \le (p-2+1+1+1)/4=(p+1)/4=m$, so using $m=|B/\lambda| \le |T_r|$, we have $|T_r|=m$. Then $B/\lambda=T_r$.
Combining this with the intersection condition \eqref{eq:n2-T-intersection-one} gives $1=|(1+T_r)\cap(1+rT_r)|$.
Since $1+t=1+rt'$ if and only if $t=rt'$, we have $|(1+T_r)\cap(1+rT_r)|=|T_r\cap rT_r|$.
Thus, putting $I_r:=|T_r\cap rT_r|$, we obtain
\begin{equation}\label{eq:n2-Ir-equals-one}
I_r=1.
\end{equation}

Suppose first that $r^2\ne1$. 
Note that 
\[
I_r=\frac18\sum_{x\in\F_p}(1-\chi(1+x))(1-\chi(1+rx))(1-\chi(1+x/r)).
\]
Indeed, since at the exceptional points, the contribution is zero by $\chi(1-r)=\chi(1-r^{-1})=1$.

Expanding the product gives
\begin{align*}
8I_r =&p-\sum_{x\in\F_p}\chi(1+x) -\sum_{x\in\F_p}\chi(1+rx) - \sum_{x\in\F_p}\chi(1+x/r) +\sum_{x\in\F_p}\chi((1+x)(1+rx))  \\ &+\sum_{x\in\F_p}\chi((1+x)(1+x/r)) +\sum_{x\in\F_p}\chi((1+rx)(1+x/r)) -S_r,
\end{align*}
where
\[
S_r := \sum_{x\in\F_p} \chi\bigl((1+x)(1+rx)(1+x/r)\bigr).
\]
The three linear character sums vanish. Moreover, using the standard identity
\[
\sum_{x\in\F_p}\chi((x-u)(x-v))=-1
\]
for $u\ne v$, together with $r^2\ne1$, we obtain
\begin{align*}
&\sum_{x\in\F_p}\chi((1+x)(1+rx))=-\chi(r)=1, \qquad  \sum_{x\in\F_p}\chi((1+x)(1+x/r)) =-\chi(r^{-1})=1, \\ &
 \qquad \qquad \qquad \text{and} \qquad \sum_{x\in\F_p}\chi((1+rx)(1+x/r))=-1,
\end{align*}
where we used $\chi(r)=-1$. Thus, we obtain
\begin{equation}\label{eq:n2-Ir-character-sum}
I_r=\frac{p+1-S_r}{8}.
\end{equation}

Combining equations \eqref{eq:n2-Ir-equals-one} and \eqref{eq:n2-Ir-character-sum}, we get $S_r=p-7$.
The polynomial $(1+x)(1+rx)(1+x/r)$ is squarefree of degree $3$. Thus Weil's bound gives $|S_r|\le2\sqrt p$.
Together with $S_r=p-7$, this implies $p-7\le2\sqrt p$.
Thus $p\le13$. Since $p=4m-1$ and $m>2$, we have $p\ge11$. Also $p\equiv3\pmod4$, so the only possibility is $p=11$.

It remains to treat the case $r^2=1$. Since $r\ne1$, we have $r=-1$. Then equation \eqref{eq:n2-r-character-conditions} gives $\chi(2)=1$. In this case,
\[
I_{-1}=|\{x\in\F_p:\chi(1+x)=\chi(1-x)=-1\}|=\frac14\sum_{x\in\F_p}(1-\chi(1+x))(1-\chi(1-x)).
\]
Indeed, at the two points $x=1$ and $x=-1$, the possible contributions are zero due to $\chi(2)=1$.
Since
\[
\sum_{x\in\F_p}\chi(1+x)=\sum_{x\in\F_p}\chi(1-x)=0
\quad \text{and} \quad \sum_{x\in\F_p}\chi((1+x)(1-x))=-\chi(-1)=1,
\]
we have $I_{-1}=\frac{p+1}{4}=m$.
Since $m>2$, this contradicts \eqref{eq:n2-Ir-equals-one}. Therefore $r=-1$ is impossible. This completes the proof.
\end{proof}

\begin{corollary}\label{prop:case2-general}
Let $p$ be an odd prime, let $G$ be a proper multiplicative subgroup of $\F_p^*$, and let $\lambda\in\F_p^*\setminus G$.
Suppose that $A,B\subseteq\F_p^*$ satisfy $|A|=2<|B|=m$, $|G|+1=2m$, and
$AB=G-\lambda$.
Then
\[
[\F_p^*:G]=2.
\]
Therefore, $G=\mathcal R_p$ and $p=11$.
\end{corollary}

\begin{proof}
Put $s:=|G|=2m-1$ and $d:=(p-1)/s=[\F_p^*:G]$.
Since $s$ is odd and $p-1$ is even, $d$ is even.

Recall that
\[
P(x)=\prod_{b\in B}(x-b) \qquad\text{and}\qquad Q(x)=\prod_{c\in J}(x+c).
\]
By \cref{lem:Y-divisibility}, $P\mid\mathcal R$, setting $n=2$,
\[
\mathcal R(x)=(8m-2)Q'(x)^2-3(6m-1)Q(x)Q''(x)+12mQ(x).
\]
Since $\deg Q=2$, we have $\deg\mathcal R\le2$.
On the other hand, $\deg P=m>2$. Thus the divisibility $P\mid\mathcal R$ forces $\mathcal R\equiv0$.

Since $Q$ is monic of degree $2$, comparison of the coefficient of
$x^2$ gives
\[
0=[x^2]\mathcal R =4(8m-2)-6(6m-1)+12m =8m-2.
\]
Since $s=2m-1$, it follows that $p\mid 4s+2$.

Suppose that $d\ge4$. Since $d$ is even, either $d=4$ or $d\ge6$.
If $d=4$, then $p=4s+1$, so $4s+2=p+1$ is nonzero modulo $p$.
If $d\ge6$, then $0<4s+2<ds+1=p$, which again shows that $4s+2$ is nonzero modulo $p$.
Both cases give a contradiction. Therefore $d=2$.

It follows that $G=\mathcal R_p$. \cref{prop:case2} gives $p=11$.
\end{proof}

\subsection{The case (3): 
$m>n>2$}

\begin{proposition}\label{prop:case3}
Let $p$ be an odd prime, let $G$ be a proper multiplicative subgroup of $\F_p^*$, and let $\lambda\in\F_p^*\setminus G$.
Suppose that $A,B\subseteq\F_p^*$ satisfy $AB=G-\lambda$.
Write $|A|=n>2$ and $|B|=m>n$.
If $|G|+n-1=nm$, then no such sets $A,B$ exist.
\end{proposition}

\begin{proof}
Suppose, for contradiction, that such sets $A,B$ exist. Recall that
\[
P(x)=\prod_{b\in B}(x-b)\qquad\text{and}\qquad Q(x)=\prod_{c\in J}(x+c).
\]
Put $s:=|G|=n(m-1)+1$ and $\kappa:=(n-1)(3mn+2m-n)$.
Also recall from \cref{lem:Y-divisibility} that 
\begin{equation}\label{eq:R-unified}
\mathcal R=\kappa Q'^2-(n+1)(3mn-n+1)QQ''+m(n-1)(n+1)(n+2)x^{n-2}Q
\end{equation}
satisfies
\begin{equation}\label{eq:P-divides-R-unified}
P\mid\mathcal R.
\end{equation}

We first show that the existence of $A,B$ is impossible when $\kappa\ne0$ in $\F_p$.
Since $Q$ is monic of degree $n$, a direct calculation gives
\begin{equation}\label{eq:R-leading-coefficient-unified}
[x^{2n-2}]\mathcal R =\kappa n^2 -(n+1)(3mn-n+1)n(n-1) +m(n-1)(n+1)(n+2) = \kappa.
\end{equation}
Thus, if $\kappa\ne0$, then $\deg\mathcal R=2n-2$  and $\operatorname{lc}(\mathcal R)=\kappa$, where $\operatorname{lc}(\mathcal R)$ is the leading coefficient of $\mathcal R$.

If $m>2n-2$, then $\deg\mathcal R=2n-2<m=\deg P$, which contradicts condition \eqref{eq:P-divides-R-unified}.
We may thus assume that $m\le2n-2$.

By equation \cref{eq:P-derivative-gaps}, $x^{n-1}\mid P'(x)$.
Define $U(x):=P'(x)/x^{n-1}$.
Then
\begin{equation}\label{eq:U-degree-leading}
\deg U=m-n \qquad\text{and}\qquad \operatorname{lc}(U)=m.
\end{equation}

By condition \eqref{eq:P-divides-R-unified}, write $\mathcal R(x)=P(x)D(x)$.
It follows from equation \eqref{eq:R-leading-coefficient-unified} that
\begin{equation}\label{eq:D-degree-leading}
\deg D=2n-2-m \qquad\text{and}\qquad \operatorname{lc}(D)=\kappa.
\end{equation}

Let $\xi$ be a root of $Q$. Since $Q$ has distinct roots, we have $Q'(\xi)\ne0$. Moreover, $Q(0)\ne0$, and so $\xi\ne0$.
Evaluating \cref{lem:X-global} at $x=\xi$, we obtain $2Q'(\xi)P'(\xi)=m(n-1)\xi^{n-2}P(\xi)$.
Using $P'(x)=x^{n-1}U(x)$, this becomes
\begin{equation}\label{eq:X-at-root-Q-unified}
2\xi Q'(\xi)U(\xi) = m(n-1)P(\xi).
\end{equation}
On the other hand, evaluating $\mathcal R=PD$ at $x=\xi$ and using definition \cref{eq:R-unified}, we obtain
\begin{equation}\label{eq:Y-at-root-Q-unified}
P(\xi)D(\xi) = \kappa Q'(\xi)^2.
\end{equation}
Multiplying equation \eqref{eq:X-at-root-Q-unified} by $D(\xi)$ and then using equation \eqref{eq:Y-at-root-Q-unified}, we get 
\[2\xi Q'(\xi)U(\xi)D(\xi) = m(n-1)\kappa Q'(\xi)^2.\]
Since $Q'(\xi)\ne0$, it follows that $2\xi U(\xi)D(\xi) = m(n-1)\kappa Q'(\xi)$.

Define $\Theta(x) := 2xU(x)D(x)-m(n-1)\kappa Q'(x)$.
Every root of $Q$ is a root of $\Theta$. Furthermore,
by conditions \eqref{eq:U-degree-leading} and \eqref{eq:D-degree-leading},
\[
\deg\Theta\le 1+(m-n)+(2n-2-m)=n-1<n=\deg Q.
\]
Since $Q$ has $n$ distinct roots, it follows that $\Theta\equiv0$.
Note that the coefficient of $x^{n-1}$ in $\Theta$ is
\[
[x^{n-1}]\Theta=2m\kappa-m(n-1)\kappa n=-m\kappa(n-2)(n+1).
\]
Since $\kappa\ne0$, and $m$, $n-2$, and $n+1$ are all nonzero due to $nm\le p-2$, $[x^{n-1}]\Theta$ is nonzero in $\F_p$. This contradicts $\Theta\equiv0$.

Now, it remains to show that $\kappa$ cannot vanish.
Put $L:=3mn+2m-n$, so that $\kappa=(n-1)L$.
Since $n-1<p$, the equality $\kappa=0$ implies $p\mid L$.

Write $p-1=ds$, and $d=[\F_p^*:G]\ge2$.
Using $s=n(m-1)+1$, we have $L-1-3s=2m+2n-4>0$.
Moreover, unless $(n,m)=(3,4)$, we have $4s-(L-1)=m(n-2)-3n+5>0$.
Indeed, if $n\ge4$, then, by the assumption, $m\ge n+1$, and so
\[
m(n-2)-3n+5\ge(n+1)(n-2)-3n+5=(n-1)(n-3)>0.
\]
If $n=3$ and $m\ge5$, then $m(n-2)-3n+5=m-4>0$.

Thus, unless $(n,m)=(3,4)$, $3s<L-1<4s$.
Since $p\ge2s+1$, it follows that $0<L<2p$.
If $p\mid L$, then necessarily $L=p$, and so
\[
3<\frac{p-1}{s}=\frac{L-1}{s}<4,
\]
contradicting the fact that $(p-1)/s=d$ is an integer.
Thus, $\kappa\ne0$ unless $(n,m)=(3,4)$.

It remains to consider $n=3$ and $m=4$.
In this case $s=10$ and $\kappa=82$.
If $\kappa\ne0$, the preceding argument already gives a contradiction. Thus we only need to consider the case $\kappa=0$. Since $p$ is odd and $p\mid82$, this forces $p=41$.

We now work in $\F_{41}$. By equation \cref{eq:delta0-P-gap-middle}, we may write $P(x)=x^4+u x^3+v$.
Write
\[
Q(x)=x^3+q_2x^2+q_1x+q_0,\qquad q_0\ne0.
\]
For $n=3$, $m=4$, and $p=41$, equation \eqref{eq:R-unified} becomes
\[
\mathcal R(x) = -13Q(x)Q''(x)-4xQ(x) = Q(x)\bigl(-13Q''(x)-4x\bigr).
\]
The polynomials $P$ and $Q$ are relatively prime. Indeed, a common root would give $b=-\lambda/a$ for some $a\in A$ and $b\in B$, and thus $ab+\lambda=0$, contrary to $AB+\lambda=G\subseteq\F_p^*$.

Since $P\mid\mathcal R$ and $\gcd(P,Q)=1$, we have $P\mid -13Q''-4x$.
The polynomial $-13Q''-4x$ has degree at most $1$, whereas $\deg P=4$.
Thus $-13Q''(x)-4x=0$.
Since $Q''(x)=6x+2q_2$, and $-13\cdot6-4=-82=0$ in $\F_{41}$, this identity gives $q_2=0$.
Thus \[Q(x)=x^3+q_1x+q_0.\]

Finally, by \cref{lem:X-global}, we have $\frac43 QP''-2Q'P'=-8xP$.
Substituting the above expressions for $P$ and $Q$, we obtain
\[
\frac43 QP''-2Q'P'+8xP=-2u x^4+8q_1x^3+(2uq_1+16q_0)x^2+(8uq_0+8v)x.
\]
Therefore $u=0$ and $q_1=0$, and finally $q_0=0$.
This contradicts $q_0=Q(0)\ne0$. Then the proposition follows.
\end{proof}

\begin{proof}[Proof of \cref{thm:delta-zero}]
Put $n=|A|$ and $m=|B|$ with $n \le m$.
If $m= n$, then \cref{prop:case1} gives a contradiction.
Thus $m>n$.
If $n>2$, then \cref{prop:case3} gives a contradiction. Thus $n=2$.
By \cref{prop:case2-general}, we have $[\F_p^*:G]=2$ and $p=11$.
Thus $G=\mathcal R_{11}$. Finally, $2m=|G|+1=6$, and thus $m=3$.
\end{proof}

\section*{Acknowledgments}
The author thanks Seoyoung Kim and Chi Hoi Yip for helpful discussions, and is especially grateful to Chi Hoi Yip for pointing out the known estimates on the sizes of the factors and for suggesting the symmetric formulation. The author was supported by the Institute for Basic Science (IBS-R029-C1). 
The author would also like to acknowledge that GPT-5.5, 5.6 Thinking were used for assistance in the writing process of this manuscript and for checking
algebraic computations. 
All proofs and their mathematical ideas were given by the author, and the author takes full responsibility for the entire content of the paper.

\bibliographystyle{abbrv}
\bibliography{references}

\end{document}